\pdfoutput=1
\documentclass[12pt]{amsart}
\usepackage{amssymb,latexsym}
\usepackage{enumerate}
\usepackage{amsfonts}
\usepackage[ansinew]{inputenc}
\usepackage{amsthm}
\usepackage{amstext}
\usepackage{amsmath}
\usepackage{amscd}
\usepackage[mathscr]{eucal}

\makeatletter
\@namedef{subjclassname@2010}{%
  \textup{2010} Mathematics Subject Classification}
\makeatother

\newtheorem{theorem}{Theorem}
\newtheorem*{lemma}{Lemma}

\numberwithin{equation}{section}

\newcommand\0{R}
\newcommand\1{r}
\newcommand\2{r_c}
\newcommand\ip[2]{\langle#1,#2\rangle}
\newcommand\4{\Delta}
\newcommand\7{d_a}
\newcommand\8{d_b}
\newcommand\9{d_c}
\renewcommand\v[1]{\mathbf{#1}}
\newcommand\eh{\frac12}

\newcommand\qq[1]{\kern1pt\overline{\kern-1pt#1\kern-0.5pt}\kern0.5pt}

\begin{document}

\baselineskip=17pt

\title[A simple proof of Feuerbach’s Theorem]
{A simple proof of Feuerbach’s Theorem}

\author[F. Hofbauer]{Franz Hofbauer}
\address{Institut f\"ur Mathematik
Universit\"at Wien, Oskar-Morgenstern-Platz 1, 1090 Wien, Austria}
\email{franz.hofbauer@univie.ac.at}

\begin{abstract} The theorem of Feuerbach states that the nine-point circle
of a nonequilateral triangle is tangent to both its incircle and its three excircles. 
We give a simple proof of this theorem.
\end{abstract}

\subjclass{51M04, 97G70}

\maketitle

Let $\triangle ABC$ be a triangle with circumradius $\0$ and sidelengths $a$, $b$ and $c$.
We choose a coordinate system, which has its origin in the circumcenter $O$ of the triangle.
We denote the vectors from $O$ to the vertices $A$, $B$ and $C$ by $\v u$, $\v v$ and $\v w$ respectively.
Then we have $a=\|\v v-\v w\|$, $b=\|\v u-\v w\|$, $c=\|\v u-\v v\|$ and 
$\|\v u\|=\|\v v\|=\|\v w\|=R$.
The semiperimeter $\eh(a+b+c)$ is denoted by $s$.

Let $N$ be the nine-point center and $H$ the orthocenter of the triangle. 
Set $\v n=\eh\v u+\eh\v v+\eh\v w$. Then we have 
$\|\v n-(\eh\v u+\eh\v v)\|=\eh\|\v w\|=\eh\0$. Similarly we get
$\|\v n-(\eh\v u+\eh\v w)\|=\eh\0$ and $\|\v n-(\eh\v v+\eh\v w)\|=\eh\0$.
Therefore $\v n$ is the vector from $O$ to $N$ und $\eh \0$ is the nine-point radius.

Let $I$ be the incenter and $I_c$ the excenter opposite $C$. 
Set $\v m=\frac a{2s}\v u+\frac b{2s}\v v+\frac c{2s}\v w$. 
The equations of the angle bisectors through $A$ and $B$ are
$$ \v u+\lambda(\tfrac1c(\v v-\v u)+\tfrac1b(\v w-\v u)) \ \ \ \text{ und }\ \ \ 
  \v v+\mu(\tfrac1c(\v u-\v v)+\tfrac1a(\v w-\v v)) $$
Choosing $\lambda=\frac{bc}{2s}$ and $\mu=\frac{ac}{2s}$ we get $\v m$ in both cases.
This shows that $\v m$ is the vector from $O$ to the incenter $I$.

Set $\v m_c=\frac a{2(s-c)}\v u+\frac b{2(s-c)}\v v-\frac c{2(s-c)}\v w$. 
The equations of the external angle bisectors through $A$ and $B$ are
$$ \v u+\lambda(\tfrac1c(\v v-\v u)-\tfrac1b(\v w-\v u)) \ \ \ \text{ und }\ \ \ 
  \v v+\mu(\tfrac1c(\v u-\v v)-\tfrac1a(\v w-\v v)) $$
Choosing $\lambda=\frac{bc}{2(s-c)}$ and $\mu=\frac{ac}{2(s-c)}$ we get $\v m_c$ 
in both cases.
This shows that $\v m_c$ is the vector from $O$ to the excenter $I_c$.

We use the following lemma for the computation of distances.

   \begin{lemma}
For any $\7$, $\8$ and $\9$ in $\mathbb R$ we have
$$\|\7\v u+\8\v v+\9\v w\|^2=\0^2(\7+\8+\9)^2-(a^2\8\9+b^2\7\9+c^2\7\8) $$
   \end{lemma}
   \begin{proof} We have 
$c^2=\|\v u-\v v\|^2=\ip{\v u-\v v}{\v u-\v v}=\|\v u\|^2+\|\v v\|^2-2\ip{\v u}{\v v}$.
Since we have also $\|\v u\|=\|\v v\|=\0$ we get $\ip{\v u}{\v v}=\0^2-\eh c^2$. 
Similarly  we get $\ip{\v u}{\v w}=\0^2-\eh b^2$ und  $\ip{\v v}{\v w}=\0^2-\eh a^2$.
Using this and $\|\v u\|=\|\v v\|=\|\v w\|=\0$ we compute
\begin{align*}
\|\7\v u+&\8\v v+\9\v w\|^2  \\
&=\7^2\|\v u\|^2+\8^2\|\v v\|^2+\9^2\|\v w\|^2 
                              +2\7\8\ip{\v u}{\v v}+2\7\9\ip{\v u}{\v w}+2\8\9\ip{\v v}{\v w} \\
   &=\0^2(\7^2+\8^2+\9^2+2\7\8+2\7\9+2\8\9)-c^2\7\8-b^2\7\9-a^2\8\9  
\end{align*}
This is the desired result.
   \end{proof}
   
We use Heron's formula $\4^2=s(s-a)(s-b)(s-c)$ for the area $\4$ of the triangle. 
This can also be written as $16\4^2=2a^2b^2+2a^2c^2+2b^2c^2-a^4-b^4-c^4$.
Furthermore, we use the formulas $\1=\frac\4s$, $\2=\frac\4{s-c}$ and $\0=\frac{abc}{4\4}$
for the inradius $\1$, the exradius $\2$ and the circumradius $\0$. 
In particular we have $\1\0=\frac{abc}{4s}$ and $\2\0=\frac{abc}{4(s-c)}$.

   \begin{theorem}\label{eul} (Euler) $|OI|^2=\0^2-2\0\1$   \end{theorem}
   \begin{proof} We have $\overrightarrow{OI}=\v m=\7\v u+\8\v v+\9\v w$ with
$\7=\frac a{2s}$, $\8=\frac b{2s}$ and $\9=\frac c{2s}$. 
We get $\7+\8+\9=1$ and $a^2\8\9+b^2\7\9+c^2\7\8=\frac{abc}{4s^2}(a+b+c)=\frac{abc}{2s}=2\1\0$.
The Lemma gives the result.
   \end{proof}

   \begin{theorem}\label{1} (Feuerbach) $|IN|=\eh\0-\1$   \end{theorem}
   \begin{proof} We have $\overrightarrow{IN}=\v n-\v m=\7\v u+\8\v v+\9\v w$ with
$\7=\frac{s-a}{2s}$, $\8=\frac{s-b}{2s}$ and $\9=\frac{s-c}{2s}$. 
We get $\7+\8+\9=\eh$ and 
\begin{align*}
a^2\8\9&+b^2\7\9+c^2\7\8  \\
    &=\tfrac{1}{16s^2}(a^2(a^2-(b-c)^2)+b^2(b^2-(a-c)^2)+c^2(c^2-(a-b)^2)) \\
     &=\tfrac{1}{16s^2}(a^4+b^4+c^4-2a^2b^2-2a^2c^2-2b^2c^2  +2abc(a+b+c))  \\
    &=-\tfrac{\4^2}{s^2}+\tfrac{abc}{4s}=-\1^2+\1\0  
\end{align*}
The Lemma implies $|IN|^2=\frac14\0^2+\1^2-\1\0=(\eh\0-\1)^2$. By Theorem~\ref{eul} we have
$\eh\0\ge\1$. Hence  $|IN|=\eh\0-\1$ follows.
   \end{proof}

   \begin{theorem}\label{2} (Feuerbach) $|I_cN|=\eh\0+\2$   \end{theorem}
   \begin{proof} We have $\overrightarrow{I_cN}=\v n-\v m_c=\7\v u+\8\v v+\9\v w$ with
$\7=\frac12-\frac a{2(s-c)}=-\frac{s-b}{2(s-c)}$, $\8=\frac12-\frac b{2(s-c)}=-\frac{s-a}{2(s-c)}$ 
and $\9=\frac12+\frac c{2(s-c)}=\frac{s}{2(s-c)}$. 
We get $\7+\8+\9=\eh$ and 
\begin{align*}
a^2\8\9&+b^2\7\9+c^2\7\8 \\
   &=\tfrac{1}{16(s-c)^2}(a^2(a^2-(b+c)^2)+b^2(b^2-(a+c)^2)+c^2(c^2-(a-b)^2)) \\
     &=\tfrac{1}{16(s-c)^2}(a^4+b^4+c^4-2a^2b^2-2a^2c^2-2b^2c^2  -2abc(a+b-c))  \\
    &=-\tfrac{\4^2}{(s-c)^2}-\tfrac{abc}{4(s-c)}=-\2^2-\2\0  
\end{align*}
The Lemma implies $|I_cN|^2=\frac14\0^2+\2^2+\2\0=(\eh\0+\2)^2$. We get $|I_cN|=\eh\0+\2$.
   \end{proof}

Theorem~\ref{1} implies that the nine-point circle is tangent to the incircle and 
Theorem~\ref{2} implies that the nine-point circle is tangent to the excircle
opposite $C$.

We get also

   \begin{theorem}  $|OH|^2=9\0^2-(a^2+b^2+c^2)$   \end{theorem}
   \begin{proof} 
Since $N$ is the midpoint between $O$ and $H$, we get that 
$\v u+\v v+\v w$ is the vector from $O$ to $H$. 
Hence the desired result follows from the Lemma with $\7=\8=\9 =1$.
\end{proof}

Classical proofs of of Feuerbach’s Theorem can be found in \cite{F} and \cite{J}. 
In \cite{S} a proof is given, which uses vector computations.
The proof in this paper is still simpler than those in \cite{S}.

\end{document}